\documentclass[12pt,a4paper]{amsart}

\usepackage{amscd,amssymb,amsthm}
\usepackage{graphicx}

\numberwithin{equation}{section}

\theoremstyle{plain}
\newtheorem{theorem}{Theorem}[section]
\newtheorem{corollary}[theorem]{Corollary}
\newtheorem{lemma}[theorem]{Lemma}
\newtheorem{proposition}[theorem]{Proposition}

\theoremstyle{definition}
\newtheorem{definition}{Definition}[section]

\theoremstyle{remark}

\numberwithin{equation}{section}

\numberwithin{table}{section}

\numberwithin{figure}{section}

\DeclareMathOperator{\lcm}{lcm}

\def\N{{\Bbb N}}

\def\1{{\bf 1}}

\def\lcm{\operatorname{lcm}}

\newcommand{\bea}{\begin{eqnarray*}}
\newcommand{\eea}{\end{eqnarray*}}
\newcommand{\bean}{\begin{eqnarray}}
\newcommand{\eean}{\end{eqnarray}}

\def\Gf#1,#2,#3,#4;{G^{#1,#2}_{#3,#4}}
\def\Hf#1,#2,#3,#4;{H^{#1,#2}_{#3,#4}}

\def\gfi#1,#2,#3,#4;{I^{(#1_i),(#2_i)}_{(#3_i),#4}}
\def\gfd#1,#2,#3,#4;{D^{(#1_i),(#2_i)}_{(#3_i),#4}}

\def\phi{\varphi}
\begin{document}
\title[ Arithmetical properties at the level of idempotence ]
{Arithmetical properties at the level of idempotence }
\author{Fethi Bouzeffour $^a$, Wissem Jedidi $^b$}
\address{$^a$ Department of mathematics, College of Sciences, King Saud University,  P. O Box 2455 Riyadh 11451, Saudi Arabia.} \email{fbouzaffour@ksu.edu.sa, wissem$_-$jedidi@yahoo.fr}
\address{$^b$ Department of Statistics \& OR, King Saud University, P.O. Box 2455, Riyadh 11451, Saudi Arabia and Universit\'e de Tunis El Manar, Facult\'e des Sciences de Tunis, LR11ES11 Laboratoire d'Analyse Math\'ematiques et Applications, 2092, Tunis, Tunisia}


\begin{abstract}In this paper we give an attempt to extend some arithmetic properties such as multiplicativity, convolution products to the setting of operators theory. We provide a significant examples which are of interest in number theory.  We also give a representation of the Euler differential operator by means of the Euler totient arithmetic function and idempotent elements of some associative unital algebra.
\end{abstract}
\maketitle
\section{Introduction}In number theory, an arithmetical, or number-theoretic function is a function $\alpha: \,\mathbb{N}\rightarrow \mathbb{C}$. Their various properties were investigated by several authors and they represent an important research topic up to now. An important property shared by many number-theoretic  functions is multiplicativity: an arithmetical function $\alpha$ is said to be multiplicative, if for all relatively prime positive integers $n,\,m$, we have $$\alpha(nm)=\alpha(n)\alpha(m).$$ \
Examples of important arithmetical functions include: the Euler totient function,  denoted $\varphi$, and  defined as the number of positive integers less than and relatively prime to $n$.  The M\"{o}bius function given by
$$ \mu(n) = \begin{cases}(-1)^{\omega(n)} & \mbox{ if } n \mbox{ is a square-free integer}\\ 0&\mbox{ otherwise },\end{cases} $$
where $\omega(n)$ is the number of prime factors of $n$ is also a multiplicative functions.  Another important multiplicative function is the Ramanujan sum's, which is defined by \cite{Ram1918}
\begin{equation*}  \label{Ramanujan_sum}
c_n(j):= \sum_{\substack{\gcd(k,n)=1\\1\leq k\leq n}} \varepsilon_n^{jk},
\end{equation*}
where $\varepsilon_n$ denotes a primitive $n$-th root of unity and  $\gcd(k,n)$ to denote the greatest common divisor of the positive integers $k$ and $n.$  These sums fit naturally with other number-theoretic  functions. For instance, one has
\begin{equation}
c_n(1)=\mu(n),\quad c_n(n)=\varphi(n),
\end{equation}
For a more elaborate account on the multiplicative functions, we refer the reader to the texts \cite{Apo,Coh1959} and the survey articles \cite{NivZucMont1991}.  \\

In this work we suggest to extend the rang of the number-theoretic function and consider functions $f$ such that:

\noindent - their domain are the positive integers and whose range is a subset of an unital associative algebra $A$ over $\mathbb{C}$;

\noindent - they satisfy the property
\begin{equation}
f(nm)=f(n)f(m), \,\, \text {when}\,\, \gcd(n,m)=1.\label{l1}
\end{equation}

Our first objective  is to give a variety of significant examples, which are of interest in number theory. Notice that if  $\mathcal{A}$ is an unital algebra then every number-theoretic function $\alpha: \mathbb{N} \rightarrow \mathbb{C}$ can be identified with the function  $n\rightarrow\alpha(n) \, e$ ($e$ is the unit of $\mathcal{A}$). \\

Recall that an element $P \in \mathcal{A}$ such that $P^2 = P$ is called idempotent. A set of idempotent elements $P_1,\dots,P_n$ is called orthogonal if, $$P_i \,P_j=\delta_{ij} P_i \quad \mbox{and}\quad \sum_{i=1}^n P_i=e.$$

In this paper, we consider a set of orthogonal idempotent $P_j(n)$, indexed by two integers  $n\geq 1$ and $j\geq 0$, satisfying:

\noindent - for every fixed integer $n,$ the sequence $j\rightarrow P_j(n)$ is periodic with period $n$;

\noindent -  for every arithmetic progression $j+n,\dots j+rn$, we have
$$P_{j}(n)=\sum_{k=1}^{r}P_{j+kn}(nr).$$
Under these conditions and for every fixed $j\geq 0$, the function $n\rightarrow P_j(n)$ is multiplicative.\\
Another example considered in this note, consists to replace the root of unity in the Ramanujan sum's by an element $s$ of $\mathcal{A}$ satisfying $s^n=e$. Then, the following  sum  \begin{align}
\sum_{\substack{\gcd(k,n)=1\\1\leq k\leq n}}s^k
\end{align}is a multiplicative function.
\section{Arithmetic properties of Idempotents}
\subsection{Multiplicative functions}
Throughout, $\mathcal{A}$ will be an associative unital algebra over $\mathbb{C}$. We recall the following definitions and basic facts
\begin{definition} A function $f :\mathbb{N}\rightarrow \mathcal{A}$ is called multiplicative if
 $$f(nm)=f(n)f(m),\,\,  \text {when}\,\, \gcd(n,m)=1.$$
\end{definition}
\begin{lemma} If $f :\mathbb{N}\rightarrow \mathcal{A}$ is multiplicative, then $f(1)$ is idempotent. \end{lemma}
\begin{proof} Taking $n=m= 1$, we have
$f(1) = f(1 . 1) = f(1) · f(1).$
So, $f(1)$ is idempotent.
\end{proof}
The following proposition is a characterization of multiplicative functions.
\begin{proposition}For an arithmetical function $f :\mathbb{N}\rightarrow \mathcal{A}$, the following are equivalent:

1. $f$ is multiplicative,

2. for all $n=p_1^{\alpha_1}\dots p_k^{\alpha_k}$ (the standard form factorization of $n$), we have
\begin{equation*}
f(n)=f(p_1^{\alpha_1})\dots f(p_k^{\alpha_k}).
\end{equation*}
\end{proposition}
We naturally  extend the convolution product such as Drichlet product, $\lcm $ product and unitary product (see, \cite{Apo,Vai1931,Siv1989}) as follows:
\begin{align}\label{conv}
&(f\, * \,g)(n)=\sum_{kl=n}f(k)g(l)\quad (\text{Dirichlet product}). \\&\label{conv1}
(f\, \Box\, g)(n)=\sum_{\lcm(k,l)=n}f(k)g(l)\quad (\text{lcm-product)}.
\\& \label{conv11}
(f\, \sqcup\, g)(n)=\sum_{kl=n\,\,\gcd(k,l)=1}f(k)g(l)\quad (\text{unitary product)}.
\end{align}
where $f,\,g: \mathbb{N}\rightarrow \mathcal{A}$.\\

The convolution products defined in \eqref{conv}, \eqref{conv1} and \eqref{conv11} are associative but not commutative in general. If, in addition, if for all integers $n,m$, and every prime numbers $p$ and $q,$  $f(p^n)$ commutes with $g(q^m)$,  then we have
$$f\, * \,g=g\, * \,f,\quad f\, \Box \,g=g\, \Box \,f,\quad f\,\sqcup \,g=g\, \sqcup \,f.$$

The following propositions are easy to prove.
\begin{proposition} 1) If $f$ is an arithmetic function on $\mathcal{A}$, then
$$f\, * \,I=I\, * \,f=f,\quad f\, \Box \,I=I\, \Box \,f=f,\quad f\,\sqcup \,I=I\, \sqcup \,f=f,$$
where
\begin{equation*}
I(n)=\left\{
  \begin{array}{l l}
    e\quad \text{if}\quad n=1 \\
   0\quad \text{otherwise}.
  \end{array}
  \right.
  \end{equation*}

2) If $f$ is an arithmetical function such that $f(1)$ is invertible element in $\mathcal{A}$, then $f$ has a unique inverse with respect to the Dirichlet product.

3) If $f$ is a multiplicative function on $\mathcal{A}$ and $b\in \mathcal{A}$ invertible element , then the function $n\mapsto bf(n)b^{-1}$ is  multiplicative.

4) If $f$ and $g$ are multiplicative functions, then $f*g$ is also a multiplicative function.
\end{proposition}
\begin{proposition}
1)If $\mathcal{A}$ is equipped with a norm $\|.\|$ and $f$ a multiplicative function on $A$, then $n\rightarrow\|f(n)\|$ is a multiplicative function.\\
2)If $\mathcal{A}$ is the algebra of $N\times N$-matrix and $f$ is multiplicative on $\mathbb{C}^N$, then function $n\mapsto \det(f(n))$ is  multiplicative.
 \end{proposition}
\subsection{Idempotents}
\begin{definition}Let $\mathcal{A}$ be an algebra with unity $e$. We say that a collection  $\,\{P_j(n)\}_{j=0,n=1}^\infty\,$ is an arithmetic system of idempotent if the following conditions are satisfied:\\

\noindent I) for all $i,\,j=1,\,2,\dots,$ we have $P_i(n)P_j(n)=\delta_{ij}P_i$(n),\\

\noindent II) for all $n,\,j=1,\,2,\dots,$ we have $P_{j+n}(n)=P_j(n),$\\

 \noindent III) for every arithmetic progression $\{j+n,\dots,j+nr\}$ we have
$$P_{j}(n)=\sum_{k=1}^{r}P_{j+kn}(nr).$$
\end{definition}
The integer $n$ in $P_j(n)$ is called level of the idempotence.
\begin{theorem}\label{th1}
Let $\{P_j(n)\}_{j=0,n=1}^\infty$ be an arithmetic system of idempotents of $\mathcal{A}$.
Then for arbitrary positive integers $n$, $m$ $k$ and $l$, we have
\begin{equation}
P_{k}(n)P_{l}(m)= \left\{
  \begin{array}{l l}
    P_{j}(\lcm(n,m))\quad \text{if}\quad \gcd(n,m)\mid l- k \\
    0\quad \text{otherwise}.
  \end{array} \right.
\end{equation}
where $j$ is the unique solution $\mod(\lcm(n,m))$ of the system of congruences
\begin{equation} \left\{
  \begin{array}{l l}
    j\equiv k \mod (n) \\
   j\equiv l \mod (m).
  \end{array} \right.
  \end{equation}
Here $\lcm(n,m)$ is the least common multiple of the integer $n,\,m.$
\end{theorem}
\begin{lemma}(The Chinese Remainder Theorem)\\
The system of congruences
\begin{equation} \left\{
  \begin{array}{l l}
    x\equiv a \mod (n) \\
   x\equiv b \mod (m).
  \end{array} \right.
  \end{equation}
is solvable if, and only if $\gcd(n,m)|a-b,$ any two solutions of the system are incongruent mod $ \lcm(n,m)$.
\end{lemma}
\begin{proof}From condition $III)$ we can write
\begin{align*}
P_{k}(n)P_{l}(m)&=\sum_{s=0}^{m-1}\sum_{r=0}^{n-1}P_{k+sn}(nm)P_{l+rm}(nm).
\end{align*}
If $\gcd(n,m)\nmid l- k$, then  $k+sn \neq l+rm$ and $P_{k}(n)P_{l}(m)=0$. \\If $\gcd(n,m)\mid l- k,$ then by Chinese remainder theorem there exist unique integer $j$ $\mod (\lcm(n,m))$ such that \begin{equation} \left\{
  \begin{array}{l l}
    j\equiv k \mod (n) \\
   j\equiv l \mod (m).
  \end{array} \right.
  \end{equation}Then from $II)$ we get
  \begin{align*}
P_{k}(n)P_{l}(m)&=\sum_{s=0}^{m-1}\sum_{r=0}^{n-1}P_{j+sn}(nm)P_{j+rm}(nm).
\end{align*}
On the other hand,  from $I)$ we see that for $0\leq r\leq n-1$ and  $0\leq s\leq m-1$ the term $P_{j+rn}(nm)P_{j+sm}(nm)$ is different from zero, if and only if $rn =sm$, which is equivalent to $r=r' \frac{n}{\gcd(n,m)}$ and $s=r' \frac{m}{\gcd(n,m)}$, with $0\leq r'\leq \gcd(n,m)-1$.\\Then
\begin{align*}
P_{k}(n)P_{l}(m)&=\sum_{r'=0}^{\gcd(n,m)-1}P_{j+r'\lcm(n,m)}(nm)=P_j(\lcm(n,m)).
\end{align*}
Therefore
\begin{align*}
P_{k}(n)P_{l}(m)=\left\{
  \begin{array}{l l}
    P_{j}(\lcm(n,m))\quad \text{if}\quad \gcd(n,m)\mid l- k \\
    0\quad \text{otherwise}
  \end{array} \right.
\end{align*}
where $j$ is the unique integer $\mod (\lcm(n,m))$ such that \begin{equation} \left\{
  \begin{array}{l l}
    j\equiv l \mod (n) \\
   j\equiv k \mod (m).
  \end{array} \right.
  \end{equation}
\end{proof}
\begin{corollary}The following hold\\
1. If $n$ and $m$ are relatively prime, we have
\begin{equation*}
P_{j}(n)P_{j}(m)=P_{j}(nm).\label{f3}
\end{equation*}
2. If $n\mid m$, we have
\begin{equation*}
P_{j}(n)P_k(m)=\left\{
  \begin{array}{l l}
    P_k(m)\,\,\, \text{if} \,\,\,k\equiv j \mod(n) \\
   0\,\,\,\text{otherwise} .
  \end{array} \right.
  \end{equation*}
\end{corollary}
\begin{proposition}Let $\alpha,\,\beta: \,\mathbb{N}\rightarrow \mathbb{C}$ two arithmetic functions and $j=0,1\dots$, we have
\begin{align}
&\alpha P_j\,\square \,\beta P_j=(\alpha \square \beta) P_j,\label{b1}\\& \alpha P_j\sqcup \beta P_j=(\alpha \sqcup \beta)\, P_j, \label{b2}\\&
\nu_0* (\alpha \square \beta )P_j=(\nu_0*\alpha P_j)\,(\nu_0*\beta P_j)\label{b3}.
\end{align}
In particular,
\begin{align}
P_j\,\square \,P_{j}(n)=M_2(n)P_{j}(n),\quad P_j\sqcup P_{j}(n)=2^{\omega(n)}P_{j}(n)
\end{align}
where
\begin{equation}
M_s(n)=\left\{
  \begin{array}{l l} 1
    \quad \text{if}\,\,\,n=1, \\
    \prod_{k=1}^r\big((a_s+1)^s-a_k^s\big)\quad\,\text{if}\quad n=\prod_{k=1}^rp_k^{a_k},
  \end{array} \right.  \label{S1}
\end{equation}
and $\nu_0(n)=1.$
\end{proposition}
\begin{proof}Let $n$ be a positive integer, we have
\begin{align*}
\alpha P_j\,\square \,\beta P_{j}(n)&=\sum_{\lcm(k,l)=n}\alpha(k)\beta (l)P_j(k)P_j(l)\\&
=\sum_{\lcm(k,l)=n}\alpha(k)\beta (l)P_j(\lcm(k,l))\\&
=(\alpha\,\square \,\beta) (n)P_{j}(n).
\end{align*}
To prove \eqref{b2}
\begin{align*}
\alpha P_j\,\sqcup \,\beta P_{j}(n)&=\sum_{kl=n, \gcd(k,l)=1}\alpha(k)\beta (l)P_j(k)P_j(l)\\&
=\sum_{kl=n, \gcd(k,l)=1}\alpha(k)\beta (l)P_{j}(n)\\&
=(\alpha\sqcup \beta) (n)P_{j}(n).
\end{align*}
The equation \eqref{b3} follows from the following identity \cite{leh}
\begin{equation}
(\nu_0\ast \alpha)(\nu_0\ast\alpha)=\nu_0\ast (\alpha\Box \beta).
\end{equation}
\end{proof}
\section{Ramanujan sum's}
The function $\alpha: \,\mathbb{N}\rightarrow\,\mathbb{C}$ is called even function ($\mod\,d$ ) if $\alpha(n) = \alpha(\gcd(n,d))$ for
all $n$, that is if the value $\alpha(n)$ depends only on the $\gcd (n, d)$. Hence if $\alpha$ is even($\mod\,d$ ), then it is sufficient to know the values $f(r)$, where $r \mid d$. \\ The Ramanujan
sum's \begin{equation}  \label{Ramanujan_sum}
c_n(j):= \sum_{\substack{\gcd(k,n)=1\\1\leq k\leq n}} \varepsilon_n^{jk}.
\end{equation}
If $\alpha$ is even ($\mod\,d$ ), then it has a Ramanujan-Fourier expansion of
the form
\begin{equation*}
\alpha(n)= \sum_{r\mid d} (\mathcal{R}_\alpha)(r) c_r(n) \qquad (n\in \N),
\end{equation*}
where the (Ramanujan-)Fourier coefficients $(\mathcal{R}_\alpha)(r)$ are
uniquely determined and given by
\begin{equation*}
(\mathcal{R}_\alpha)(r)=  \sum_{\delta\mid d} \alpha(\tfrac{d}{\delta})c_{\delta}(\tfrac{d}{r}).
\end{equation*}
\begin{definition}We say that a set  $\{P_j(n): n\geq 1,j\geq 0\}$ is an arithmetic system of orthogonal idempotent, if\\
1. $\{P_j(n)\}_{j,n=1}^\infty$ is an arithmetic system of idempotent,\\
2. for every $n$
\begin{equation}
\sum_{j=0}^{n-1}P_j(n)=e.
\end{equation}
\end{definition}
Consider the sum
\begin{align}
C_j(n)=\sum_{\substack{\gcd(k,n)=1\\1\leq k\leq n}}\varepsilon_n^{-jk}S^{k}(n),\quad j=0,\,1,\,\dots,
\end{align}
where $
S(n)=\sum_{j=1}^{n}\varepsilon_n^{j}P_j(n).$\\
For every divisor $r$ of $n$, we denote by
\begin{align}
T_{r,j}(n)=\sum_{\substack{\gcd(k,n)=\tfrac{n}{r}\\1\leq k\leq n}}P_{k+j}(n).
\end{align}
Note that $T_{n,j}(n)$ is denoted simply by $T_{j}(n)$.
\begin{proposition}The following properties hold
\begin{equation*}
C_j(n)=\mu*\nu_1P_j(n),
\quad \& \quad
T_j(n)=\nu_0*\mu P_j(n),
\end{equation*}
where  $\nu_k(n)=n^k.$
\end{proposition}
\begin{proof} We have
\begin{align*}
\mu*\nu_1 P_j(n)&=\sum_{d\mid n}\mu(d)\tfrac{n}{d}P_j(\tfrac{n}{d})\\&=\sum_{d\mid n}\mu(d)\sum_{k=1}^{\tfrac{n}{d}}\varepsilon_n^{-jkd}S_n^{kd}\\&=
\sum_{k=1}^n\varepsilon_n^{-jk}S_n^{k}\sum_{\substack{d\mid n\\d\mid k} }\mu(d)\\&
=\sum_{k=1}^n\varepsilon_n^{-jk}S_n^{k}\sum_{\substack{d\mid \gcd(k,n)} }\mu(d).
\end{align*} From Theorem 2.1 in \cite{Apo} we have
\begin{equation}
\sum_{d \mid \gcd(k,n)}\mu(r)=\left\{
  \begin{array}{l l}
    1\quad\, \text{if}\quad \gcd(k,n)=1, \\
   0\quad\, \text{if}\quad \gcd(k,n)\,>1.
  \end{array} \right.  \label{S1}
\end{equation}
Hence
\begin{equation}
C_j(n)=\mu*\nu_1 P_j(n).
\end{equation}
To prove 2) \begin{align*}
T_{j}(n)&=\sum_{\substack{\gcd(k,n)=1\\1\leq k\leq n}}P_{j+k}(n)\\&=\sum_{k=1}^n\sum_{\substack{\delta \mid \gcd(k,n)}}\mu(\delta)P_{j+k}(n)\\&=
\sum_{\substack{\delta \mid n}}\mu(\delta)\sum_{k=1}^{\tfrac{n}{\delta}} P_{j+k\delta}(n)
\\&=
\sum_{\substack{\delta \mid n}}\mu(\delta) P_{j}(\delta)
\end{align*}
\end{proof}
\begin{corollary}
\begin{align*}
&C_j(n)=n\prod_{l=1}^N\big(P_j(p^{\alpha_l})-
\frac{1}{p}P_j(p^{\alpha_l-1}),\\&T_j(n)=\prod_{\substack{p\mid n\\p \,\,\text{prime}}}(e-P_j(p)).
\end{align*}
\end{corollary}
\begin{proof} The functions $C_j(n)$ and $T_j$ are multiplicative. So it suffice to compute $C_j(p^k)$ and $T_j(p^k)$ where $p$ is prime number.
Form proposition 4.1, we have
\begin{align*}
C_j(p^k)&=\sum_{l=0}^k\mu(p^l)p^{k-l}P_j(p^{k-l})\\&=p^{k}P_j(p^{k})-p^{k-1}P_j(p^{k-l})\\&
=p^{k}\big(P_j(p^{k})-p^{-1}P_j(p^{k-1})\big),
\end{align*}
and \begin{equation*}
T_j(p^k)=\left\{
  \begin{array}{l l}
    e-P_j(p)\quad\, \text{if}\quad k=1, \\
   0\quad\, \text{otherwise}.
  \end{array} \right.
\end{equation*}
Since $C_j$ is multiplicative we have \begin{align*}
C_j(n)&=\prod_{l=1}^NC_j(p^{\alpha_l})=n\prod_{l=1}^N\big(P_j(p^{\alpha_l})-
\frac{1}{p}P_j(p^{\alpha_l-1}).
\end{align*}
Similarly,
\begin{equation*}
T_j(n)=\prod_{\substack{p\mid n\\p \,\,\text{prime}}}(e-P_j(p))
\end{equation*}
\end{proof}
\begin{theorem}
For fixed integers $n$ and $j$,  $\{T_{r,j}(n): r\mid n\}$ is a set of $\tau(n)$ orthgonal idempotent:
\begin{equation}\sum_{r\mid n}T_{r,j}(n)=e\quad \text{and} \quad
T_{r,j}(n)T_{r',j}(n)=\delta_{r,r'}T_{r,j}(n)
\end{equation}
where $\tau(n)$ is the number of divisor of $n.$
\end{theorem}
\begin{lemma}
The operators $T_j(n)$ and $C_j(n)$ are related by
$$ C_j(n)=\sum_{r\mid n} c_{n}(n/r)T_{r,j}(n),\quad T_{n,j}(n)=\frac{1}{n}\sum_{r\mid n} c_{n}(n/r)\,C_{j}(r).$$
\end{lemma}
\begin{proof}
\begin{align}
\sum_{r\mid n} c_{n}(n/r)T_{r,j}(n)&
=\sum_{r\mid n} c_{n}(n/r)\sum_{\gcd(k,n)=\tfrac{n}{r}}P_{k+j}(n)\\&=
\sum_{l=1}^n\sum_{r\mid n} c_{n}(n/r)\sum_{\gcd(k,n)=\tfrac{n}{r}}\varepsilon_n^{-(k+j)l}S_n^l
\\&=\frac{1}{n}
\sum_{l=1}^n\varepsilon_n^{-lj}\sum_{r\mid n} c_{n}(n/r)c_r(l)S_n^l
\end{align}
From formula Exercise 2.23 in \cite{McC1986}, we have
\begin{equation}
\sum_{r\mid n} c_{n}(n/r)c_r(l)=\left\{
  \begin{array}{l l}
    n\quad\, \text{if}\quad \gcd(l,n)=1, \\
   0\quad \text{if}\quad \gcd(l,n)>1.
  \end{array} \right.  \label{SSys1}
\end{equation}
Hence
\begin{align}
\sum_{r\mid n} c_{n}(n/r)T_{r,j}(r)
=\sum_{\gcd(l,n)=1} \varepsilon_n^{-lj}S_{n}^l= C_j(n).
\end{align}
Similarly,
\begin{align}
\sum_{r\mid n} c_{n}(n/r)C_{j}(r)&
=\sum_{r\mid n} c_{n}(n/r)\sum_{\gcd(k,n)=\frac{n}{r}}\varepsilon_n^{-jk}S_n^{k}\\&=
\sum_{l=1}^n\sum_{r\mid n} c_{n}(n/r)c_r(l-j)P_{l}(n)
\\&=
\sum_{l=1}^n\sum_{r\mid n} c_{n}(n/r)c_r(l)P_{l+j}(n)
\end{align}
The result follows from formula (see, Exercise 2.23 \cite{McC1986})
\begin{equation}
\sum_{r\mid n} c_{n}(\tfrac{n}{r})c_r(k)=\left\{
  \begin{array}{l l}
    n\quad\, \text{if}\quad \gcd(k,n)=1, \\
   0\quad \text{if}\quad \gcd(k,n)>1.
  \end{array} \right.  \label{SSys1}
\end{equation}
Hence
\begin{align}
\sum_{r\mid n} c_{n}(n/r)C_{j}(r)=n\sum_{\gcd(k,n)=1}  P_{k+j}(n)=nT_{n,j}(n).
\end{align}
\end{proof}
\begin{theorem}Let $\alpha :\mathbb{N}\rightarrow \mathbb{C}$ be even function $(mod \,\,n)$, then for all $j=0,\,1,\,\dots, $, we have
\begin{equation}
\sum_{r\mid n}\alpha(n/r) C_j(r)=\sum_{r\mid n}\mathcal{R}(\alpha)(r)T_{r,j}(n).
\end{equation}
\end{theorem}
\begin{proof}
\begin{align*}
\sum_{r\mid n}\alpha(r) C_j(r)&=\sum_{r\mid n}\sum_{\delta\mid r}c_r(r/\delta) T_{\delta,j}(r)\, \alpha(n/r)\\&
=\sum_{r\mid n}\sum_{\delta\mid r}c_r(r/\delta) \sum_{\substack{\gcd(k,r)=\tfrac{r}{\delta}\\1\leq k\leq r}}P_{k+j}(r)\, \alpha(n/r)\\&
=\sum_{r\mid n}\alpha(n/r)\sum_{k=1}^rc_r(k) P_{k+j}(r)
\end{align*}
The Ramanujan sum $c_r$ is periodic function with period equal to $r$, then from condition $III)$ of the Definition 2.2, we can write
\begin{align*}
\sum_{k=1}^rc_r(k) P_{k+j}(r)&=\sum_{k=1}^r\sum_{l=0}^{n/r-1}c_r(k+lr) P_{k+lr+j}(r)
\\&=\sum_{k=1}^nc_r(k) P_{k+j}(n)\\&
=\sum_{\delta \mid n}c_r(n/\delta) \sum_{\substack{\gcd(k,n)=\tfrac{n}{\delta}\\1\leq k\leq k}}P_{k+j}(n)\\&=\sum_{\delta \mid n}c_r(n/\delta)T_{\delta,j}(n).
\end{align*}
Hence
\begin{align*}\sum_{r\mid n}\alpha(r) C_j(r)&
=\sum_{\delta\mid n}\sum_{r\mid n}c_r(n/\delta)\alpha(n/r)T_{\delta,j}(n)
\\&
=\sum_{\delta\mid n}\mathcal{R}(\alpha)(\delta)T_{\delta,j}(n).
\end{align*}
\end{proof}
\section{Example}
We denote by $H_R$  the vector space of all analytic functions on the open ball
$B(0,R)$ in complex plane $\mathbb{C}$. $H_R$ endowed with the topology of compact convergence, it is a complete locally convex topological vector spaces.

\noindent For $n=2,\,3,\,\dots,$ we denote by $S(n)$ the diagonal operator acting on monomials $e_k(z)=z^k, \,(k=0,\,1,\dots\,)$ as follows
$$S(n)e_k=\varepsilon^k_n e_k,\quad\varepsilon_n=e^{\tfrac{2i\pi}{n}}.$$
Its clearly that $S(n)$ is a continuous map from  $H_R$ into its self satisfying $S^n(n)=i_{H_R}$. The primitive idempotents $P_{1}(n),\,\dots P_{n}(n)$  related to  $S(n)$ are given by
\begin{equation}
P_{j}(n)=\frac{1}{n}\sum_{l=0}^{n-1}\varepsilon_n^{-lj}S^l(n).\label{D1}
\end{equation}
These obey at the following relations
\begin{equation}
i_{H_R}=\sum_{j=1}^{n}P_{j}(n),\,\,\,\,\,P_{i}(n)P_{j}(n)=\delta_{ij}P_{j}(n).
\end{equation}
The set $\{P_j(n)\}_{j,n=1}^\infty$ is an arithmetic system of orthogonal idempotent on the algebra $L(H_R)$ of all continuous linear maps of the space $H_R.$ To see this it suffice to prove for all positive integers $n$ and $m$ such that $n$ divide $ m$ the following identity
$$P_{j}(n)=\sum_{k=1}^{m/n}P_{j+kn}(m).$$
Indeed, from \eqref{D1} and the following formula
\begin{equation}
\sum_{k=1}^{r}
\varepsilon_{r}^{-lk}=\left\{
  \begin{array}{l l}
    r, \,\, \text{if} \,\, r | l \\
   0,\,\,\,\text{otherwise} ,
  \end{array} \right.\label{S1}
\end{equation}
we can write
\begin{align*}
\sum_{k=1}^{m/n}P_{j+kr}(m)&=\frac{1}{m}\sum_{k=1}^{m/n}\sum_{l=1}^m
\varepsilon_{m}^{-l(j+kr)}S^l(m)\\&=\frac{1}{m}\sum_{l=1}^d\varepsilon_m^{-lj}S^l(m)
\sum_{k=1}^{m/n}
\varepsilon_{m}^{-lkr}
\\&=\frac{1}{n}\sum_{l=1}^r\varepsilon_{n}^{-lj}S^l(n)\\&=P_{j}(n).
\end{align*}
The operators $C_0(n)$ and $T_0(n)$ acting on the basis $\{e_k\}$ as follows
\begin{align}
&C_0(n)e_m=c_n(m)e_m,\\&
T_0(n)e_m=\left\{
  \begin{array}{l l}
    e_m\quad\, \text{if} \,\,\gcd(n,m)=1, \\
   0\quad\,\text{otherwise }.
  \end{array} \right.
\end{align}
Indeed, from Proposition 3.1, we have
\begin{align*}
T_0(n)e_m&=\sum_{d\mid n}\mu(d)P_0(d)e_m\\&=\sum_{d\mid \gcd(n,m)}\mu(d)\,e_m\\&=
\left\{
  \begin{array}{l l}
    e_m\quad\, \text{if} \,\,\gcd(n,m)=1, \\
   0\quad\,\text{otherwise }.
  \end{array} \right.
\end{align*}
Now, if we compare the trace and the determinant of the matrix $[T_0(n)]_N$(resp. $[C_{0}(n)]_N$) of the restriction of $C_{0}(n)$ (resp. $T_{0}(n)$) to the subspace generated by $\{e_1,\dots e_N\}$, we get
\begin{equation}
\prod_{k=1}^Nc_n(k)=\left\{
  \begin{array}{l l} \prod_{p|n\,\text{prime}}(1-p)^{[\tfrac{N}{p}]}
    \quad \text{if}\,\,\,n\,\text{is squarefree}, \\
   0\quad\, \text{if}\quad \text{otherwise}.
  \end{array} \right.  \label{S1}
\end{equation}
\begin{equation}
\sum_{\substack{\gcd(k,n)=1\\1\leq k\leq n} } [\tfrac{N-k}{n}]=N \omega(n)-\sum_{p|n\,\,\text{prime}}[\tfrac{N}{p}]=\sum_{r\mid n} \mu(r)[\tfrac{N}{r}],
\end{equation}
and
\begin{equation}
\sum_{k=1}^Nc_n(k)=\sum_{l=1}^k\big(p_l^{\alpha_l}[\tfrac{N}{p_l^{\alpha_l}}]
-p_l^{\alpha_l-1}[\tfrac{N}{p_l^{\alpha_l-1}}]\big)=\sum_{d\mid n}d\mu(n/d)[\tfrac{N}{d}].
\end{equation}
Let $H_0$ the subspace of all function $f\in H_R$ such that $f(0)=0$ and let $\alpha: \mathbb{N}\rightarrow \mathbb{C}$. For $f\in H_0$, we put
\begin{equation}
\mathcal{P}(\alpha)f=\sum_{n=1}
^\infty \alpha(n)P_0(n)f. \label{op}
\end{equation}
\begin{proposition}
 Let $\alpha$ be an arithmetic function. Then  $\mathcal{P}(\alpha)$ defines a continuous diagonal map form $H_0$ into itself, if and only if $\limsup_{m\rightarrow \infty}|(\nu_0*\alpha)(m)|^{1/m}\leq 1$.\\Furthermore,
 \begin{align}
\mathcal{P}(\epsilon)=i_{H_0}, \quad \mathcal{P}(\alpha\Box \beta)= \mathcal{P}(\alpha)\mathcal{P}(\beta).\label{box1}
\end{align}
where
\begin{equation}
\epsilon(n)=\left\{
  \begin{array}{l l}
    1\quad \text{if}\quad n=1 \\
   0\quad \text{otherwise}.
  \end{array}
  \right.
  \end{equation}
\end{proposition}
\begin{proof}
The operator $\mathcal{P}(\alpha)$ acts on monomials $z^m$ as follows
\begin{equation}
\mathcal{P}(\alpha)z^m=\left\{
  \begin{array}{l l} \sum_{n=1}^\infty \alpha(n)
    \quad \text{if}\,\,\,m=0, \\
   (\nu_0*\alpha)(m)z^m\quad \text{otherwise}.
  \end{array} \right.  \label{S1}
\end{equation}
Suppose $\mathcal{P}(\alpha)$ defines a continuous linear map on  $H_0$. It is well-known that $$f(z)=\sum_{m=1}^\infty a_mz^m\in H_R \quad\text{iff} \quad \limsup_{m\rightarrow \infty}|a_m|^{1/m}\leq R.$$Hence,
\begin{align*}
\sum_{m=1}^\infty (\nu_0*\alpha)(m)a_mz^m \in H_R &\quad \Leftrightarrow\quad  \limsup_{m\rightarrow \infty}|a_m(\nu_0*\alpha)(m)|^{1/m}\leq R \\&\quad \Leftrightarrow\quad \limsup_{m\rightarrow \infty}|(\nu_0*\alpha)(m)|^{1/m}\leq 1.
\end{align*}
This show that $$ \limsup_{m\rightarrow \infty}|\nu_0*\alpha)(m)|^{1/m}\leq 1.$$ The converse follows from the Closed Graph Theorem (see, \cite{Rud} Theorem
2.15, p. 51]).
To prove \eqref{box1}, we use equation \eqref{b1}
\begin{align*}
\mathcal{P}(\alpha \Box\beta)&=\sum_{n=1}
^\infty (\alpha \Box \beta)(n)P_0(n)\\&=
\sum_{n=1}
^\infty (\alpha P_0\Box  \beta P_0)(n)\\&=
\sum_{n=1}
^\infty \alpha(n) P_0(n)\sum_{n=1}^\infty \beta (n)P_0(n)
\\&=\mathcal{P}(\alpha)\mathcal{P}(\beta)
\end{align*}
\end{proof}
Many examples of operators acting on the space $H_0$ such as the Euler operator,  the integration operator, the shift operator and the backward shift operator given by
$$(\theta f)(z)=zf'(z), \quad (If)(z)=\int_0^zf(t)\,dt, \quad (Uf)=zf(z), \quad (U^*f)(z)=\frac{f(z)-f(0)}{z},$$
can be represented by means of the map $\mathcal{P}(\alpha)$ for suitable arithmetic function $\alpha$. For $\alpha=\varphi$ (where $\varphi$ is the Euler's totient function), we have
\begin{align}
\mathcal{P}(\phi)e_m=\nu_0*(\mu*\nu_1)(m)e_m=\nu_1(m)e_m=me_m.\label{coi}
\end{align}
On the other hand, we have
$$ \limsup_{m\rightarrow \infty}|\nu_0*\phi)(m)|^{1/m}= \limsup_{m\rightarrow \infty}m^{1/m}= 1.$$
Then, $\mathcal{P}(\phi)$ is a continuous map on $H_0$.  From \eqref{coi}, the mapping $\mathcal{P}(\phi)$ coincides with the Euler operator $\theta =z\frac{d}{dz}$ on the monomials, hence
\begin{equation}
\mathcal{P}(\phi)=\theta.
\end{equation}
Moreover, from Proposition 4.1 and the following identity  \cite{McC1986}
$$J_r=\underbrace{\phi\,\Box\,\dots \,\Box\, \phi}_{r\text{-times}},$$
we get
\begin{align*}
\mathcal{P}(J_r)=\mathcal{P}(\phi)\dots \mathcal{P}(\phi)=\theta^r.
\end{align*}
($J_r$ called Jordan function)\\
A similarly arguments show that
$$\mathcal{P}(\mu)=e_1\otimes e_1, \quad
\mathcal{P}(\mu\ast v_1)=IU^*.$$


\begin{thebibliography}{99}
\bibitem{Apo} T.~M.~Apostol, {\it Introduction to Analytic Number Theory}, Springer, 1976.

\bibitem{Coh1959} E.~Cohen, Representations of even functions (mod $r$), II. Cauchy products, {\it Duke Math. J.} {\bf 26} (1959), 165--182.

\bibitem{McC1986} P.~J.~McCarthy, {\it Introduction to Arithmetical Functions}, Uni\-ver\-si\-text, Springer, 1986.

\bibitem{Nag} N.I. Nagnibida, P.P. Nastasiev, Strongly cyclic elements of the diagonal operator and completeness of one system of analytic functions, Dokl. Akad. Nauk Ukraine SSR Ser. A (1983) 5--8.

\bibitem{Nat1994} M.~B.~Nathanson, {\it Additive Number Theory. The Classical Bases}, Graduate Texts in Mathematics 164, Springer, 1996.

\bibitem{NivZucMont1991} I.~Niven, H.~S.~Zuckerman, H.~L.~Montgomery {\it An Introduction to the Theory of Numbers}, 5th edition, John Wiley \& Sons, 1991.

\bibitem{Ore1988} {\O}.~Ore, {\it Number Theory and Its History}, Dover Publications, 1988.

\bibitem{Ram1918} S.~Ramanujan, On certain trigonometrical sums and their applications in the theory of numbers, {\it Trans. Cambridge Philos. Soc.} {\bf 22} (1918), 259--276 (Collected Papers, Cambridge 1927, No. 21).

\bibitem{leh}D. H. Lehmer, On a theorem of von Sterneck, Bull. Amer. Math. Soc., 37 (1931), no. 10, 723--726.

\bibitem{Rud} W. Rudin, Functional Analysis, McGraw--Hill, 1991.

\bibitem{SchSpi1994} W.~Schwarz, J.~Spilker, {\it Arithmetical Functions}, London Mathematical Society Lecture Note Series, 184, Cambridge University Press, 1994.

\bibitem{Siv1989} R.~Sivaramakrishnan, {\it Classical Theory of Arithmetic Functions}, in Monographs and Textbooks in Pure and
Applied Mathematics, Vol.\ 126, Marcel Dekker, 1989.

\bibitem{TotHau2010} L.~T\'oth, P.~Haukkanen, The discrete Fourier transform of $r$-even functions, {\it Acta Univ. Sapientiae, Mathematica} {\bf 3}, no. 1 (2011), 5--25.

\bibitem{Vai1931} R.~Vaidyanathaswamy, {\it The theory of multiplicative arithmetic functions}, Trans. Amer. Math. Soc. {\bf 33} (1931), 579--662.
\end{thebibliography}
\end{document}